\documentclass[a4paper,12pt]{article}

\usepackage{hyperref}
\usepackage[pdftex]{graphicx}
\usepackage[section] {placeins}
\usepackage{enumerate}
\usepackage[normalem]{ulem}
\usepackage{amsmath}
\usepackage{amsthm}
\usepackage{latexsym}
\usepackage{amsfonts}
\usepackage{amssymb}
\usepackage{mathtools}
\usepackage{longtable}

\usepackage{tikz-cd}
\usetikzlibrary{graphs}
\usetikzlibrary{arrows.meta}
\usetikzlibrary{shapes.geometric}

\usepackage{authblk}

\usepackage{latexsym,amssymb,amsfonts, amsthm, amsmath}
\usepackage[english]{babel}
\usepackage{bm}
\usepackage{mathrsfs}
\usepackage{algpseudocode}
\usepackage{algorithm}

\algnewcommand{\IIf}[1]{\State\algorithmicif\ #1\ \algorithmicthen}
\algnewcommand{\EndIIf}{\unskip\ \algorithmicend\ \algorithmicif}

\newtheorem{thm}{Theorem}[section]

\newtheorem{prop}[thm]{Proposition}

\newtheorem*{thm*}{Theorem \ref{th:main3}}
\newtheorem*{thm*2}{Theorem \ref{th:main1}}

\theoremstyle{definition}

\newtheorem{rem}[thm]{Remark}

\newcommand{\Z}{\mathbb{Z}}

\usepackage{fullpage}

\makeatletter
\newcommand{\subjclass}[2][1991]{%
\let\@oldtitle\@title%
\gdef\@title{\@oldtitle\footnotetext{#1 \emph{Mathematics subject classification.} #2}}%
}
\newcommand{\keywords}[1]{%
\let\@@oldtitle\@title%
\gdef\@title{\@@oldtitle\footnotetext{\emph{Key words and phrases.} #1.}}%
}
\makeatother

\begin{document}

\title{Alternating Parity Weak Sequencing}
\date{}
\author[1]{Simone Costa}
\author[2]{Stefano Della Fiore}

\affil[1]{DICATAM, Sez.~Matematica, Universit\`a degli Studi di Brescia, Via Branze~43, I~25123 Brescia, Italy}
\affil[2]{DI, Universit\`a degli Studi di Salerno, Via Giovanni Paolo II 132, 84084 Fisciano, Italy}

\affil[ ]{\texttt {simone.costa@unibs.it, s.dellafiore001@unibs.it}}


\maketitle
\begin{abstract}
A subset $S$ of a group $(G,+)$ is $t$-{\em weakly sequenceable} if there is an ordering $(y_1, \ldots, y_k)$ of its elements such that the partial sums~$s_0, s_1, \ldots, s_k$, given by $s_0 = 0$ and $s_i = \sum_{j=1}^i y_j$ for $1 \leq i \leq k$, satisfy $s_i \neq s_j$ whenever and $1 \leq |i-j|\leq t$.

In \cite{CDO} it was proved that if the
order of a group is $pe$ then all sufficiently large subsets of the non-identity elements are $t$-weakly sequenceable when $p > 3$ is prime, $e \leq 3$ and $t \leq 6$. Inspired by this result, we show that, if $G$ is the semidirect product of $\mathbb{Z}_p$ and $\mathbb{Z}_2$ and the subset $S$ is balanced, then $S$ admits, regardless of its size, an \emph{alternating parity} $t$-weak sequencing whenever $p > 3$ is prime and $t \leq 8$. A subset of $G$ is balanced if it contains the same number of even elements and odd elements and an alternating parity ordering alternates even and odd elements.

Then using a hybrid approach that combines both Ramsey theory and the probabilistic method we also prove, for groups $G$ that are semidirect products of a generic (non necessarily abelian) group $N$ and $\mathbb{Z}_2$, that all sufficiently large balanced subsets of the non-identity elements admit an alternating parity $t$-weak sequencing.

The same procedure works also for studying the weak sequenceability for generic sufficiently large (not necessarily balanced) sets. Here we have been able to prove that, if the size of a subset $S$ of a group $G$ is large enough and if $S$ does not contain $0$, then $S$ is $t$-weakly sequenceable.
\end{abstract}
\section{Introduction}
A subset $S$ of a group $(G, +)$ is {\em sequenceable} if there is an ordering $(y_1, \ldots, y_k)$ of its elements such that the partial sums~$s_0, s_1, \ldots, s_k$, given by $s_0 = 0$ and $s_i = \sum_{j=1}^i y_j$ for $1 \leq i \leq k$, are distinct, with the possible exception that we may have~$s_k = s_0 = 0$.
Throughout the paper, we will use the additive notation also for non-abelian groups, i.e. $(G,+_G)$, and we avoid using the subscript in $+_G$ if it is clear from the context.

Several conjectures and questions concerning the sequenceability of subsets of groups arose in the Design Theory context: indeed this problem is related to Heffter Arrays and $G$-regular Graph Decompositions, see \cite{A15,ADMS16, CMPP18, OllisSurvey,PD}. Alspach and Liversidge combined and summarized many of them into the conjecture that if a subset of an abelian group does not contain $0$ then it is sequenceable (see \cite{AL20}).

Note that, given a sequencing of a set $S$, if $s_k\not=0$, then the partial sums define a simple path $P=(s_0,s_1,\ldots,s_k)$ in the Cayley graph $Cay[G:\pm S]$ such that $\Delta(P) = \pm S$ where $\Delta(P)$ is the multiset $\pm [s_1-s_0,s_2-s_1,\ldots,s_k-s_{k-1}]$. Here $Cay[G:\pm S]$ is the graph whose vertex set is $G$ and whose edges are the pairs $\{x,y\}$ such that $x-y\in \pm S$. Inspired by this interpretation, in \cite{CD} the authors proposed a weakening of this concept.  In particular they want to find an ordering of a subset $S$ of $G\setminus\{0\}$ whose partial sums $s_i$ and $s_j$ are different whenever $1 \leq |i-j|\leq t$. If such an ordering exists, we say that $S$ is $t$-{\em weakly sequenceable}. In this case, given a $t$-weak sequencing of a set $S$ and assuming again that $s_k\not=0$, the partial sums define a walk $W$, in the Cayley graph $Cay[G:\pm S]$, of girth bigger than $t$ (for a given $t < k$) and such that $\Delta(W) = \pm S$.

In \cite{CDO} the authors, using a polynomial approach, proved that:
\begin{thm}\label{th:main2}
Let $n = pe$ with~$p>3$ prime and let $G$ be a group of size $n$. Then subsets~$S$ of size~$k$ of~$G\setminus\{0\}$ are $t$-weakly sequenceable whenever $e\in \{1,2,3\}$, $k$ is large enough and $t\leq 6$.
\end{thm}
In this paper, we show that, for particular kinds of sets, this result can be improved. More precisely, we will consider subsets $S$ of a group $(G, +_G)$ that is the semidirect product of $(N,+_N)$ and $(\mathbb{Z}_2,+_{\Z_2})$. For such groups we will use the notation $G=N \rtimes_{\varphi} \mathbb{Z}_2$ for some group homomorphism $\varphi: \mathbb{Z}_2 \rightarrow Aut(N)$. We will indicate with the pair $(x,a)$, where $x\in N$ and $a\in \Z_2$, a generic element of $G$. Using these notations, we recall that the operation of $G$ is defined so that $(x_1,a_1)+_G(x_2,a_2)=(x_1 +_N \varphi(a_1)(x_2), a_1 +_{\mathbb{Z}_2} a_2)$,  where $\varphi(a_1)$ is the automorphism of $N$ associated to $a_1$.

Let $G=N \rtimes_{\varphi} \mathbb{Z}_2$, we define the {\em type} of~$S \subseteq G$ to be the sequence $\bm{\lambda} = (\lambda_0, \lambda_1)$, where $\lambda_0 = |S \cap N \rtimes_\varphi \{0\}|$ and $\lambda_1 = |S| - \lambda_0$.  Here the elements in $ N \rtimes_\varphi \{0\}$ are called the even elements of $G$ while the ones in $ N \rtimes_\varphi \{1\}$ are called odd. Then we say that $S$ is {\em balanced} if it has even size $k$ and it has type $(k/2, k/2)$ which means that $S$ has the same number of even and odd elements.  For such sets, it is natural to look for \emph{alternating parity} orderings. These are orderings that alternate even and odd elements and have been studied in the literature both in relation to sequencing problems (see Example $2$ of \cite{OllisSurvey} and the subsequent discussion) and in relation to some enumeration problems (see, for instance \cite{GJ}).
For this reason, on balanced sets, we will look for alternating parity $t$-weak sequencings.

In this setting, taking $N = \Z_p$, we have been able to prove the following first improvement on Theorem~\ref{th:main2}.
\begin{thm}\label{th:pol1}
Let $p$ be an odd prime, let $G=\mathbb{Z}_p \rtimes_{\varphi} \mathbb{Z}_2$ and let $S$ be a balanced subset of $G\setminus \{(0,0)\}$.
Then $S$ admits an alternating parity $t$-weak sequencing whenever $p > 3$ is prime and $t \leq 8$.
\end{thm}
Note that in this result we do not have to require that the size of $S$ is sufficiently large. On the other hand, using a hybrid approach that combines both Ramsey theory and the probabilistic method, we have been able to prove also the following asymptotic result which is also an improvement on Theorem \ref{th:main2}.
\begin{thm}\label{th:main3}
Let $G = N \rtimes_{\varphi} \mathbb{Z}_2$. Then balanced subsets of size $k$ of $G\setminus \{(0,0)\}$ admit an alternating parity $t$-weak sequencing whenever $k$ is large enough with respect to $t$.
\end{thm}
The same procedure works also for studying the weak sequenceability for generic sufficiently large (not necessarily balanced) sets. In this case, we have been able to prove what is, perhaps, the main result of this paper.
\begin{thm}\label{th:main1}
Let $G$ be a group of size $n$. Then subsets of size~$k$ of~$G\setminus\{0\}$ are $t$-weakly sequenceable whenever $k$ is large enough with respect to $t$.
\end{thm}
In particular, this theorem extends the probabilistic results presented in Section 4 of~\cite{CD}.

The paper is organized as follows.
In Section $2$ we will revisit the polynomial method and apply it in case of balanced sets in order to obtain Theorem \ref{th:pol1}.
Then, in Section $3$ are presented our asymptotic results in the case of generic and in that of balanced sets. We first show a Ramsey theory-kind lemma and then, using the probabilistic method, we prove the main results of this paper i.e. Theorems \ref{th:main3} and \ref{th:main1}.

Finally, in the last section, we prove, using a direct construction, that balanced subsets of $N \rtimes_{\varphi}\mathbb{Z}_2 \setminus \{(0,0)\}$ are alternating parity $3$-weakly sequenceable. If $k\equiv 8\pmod{10}$ and $N$ has an odd size, we can improve this result by proving that any balanced set admits an alternating parity $4$-weak sequencing.

\section{Polynomial method and Balanced Sets}\label{sec:poly}
In this section, we apply a method that relies on the Non-Vanishing Corollary of the Combinatorial Nullstellensatz, see~\cite{Alon99,Michalek10}. Given a prime $p$ (in the following $p$ will be always assumed to be a prime), this corollary allows us to obtain a nonzero point to suitable polynomials on $\mathbb{Z}_p$ derived starting from the ones defined in \cite{HOS19} and in \cite{Ollis}.
\begin{thm}\label{th:pm}{\rm (Non-Vanishing Corollary)}
Let~$\mathbb{F}$ be a field, and let $ f(x_1, x_2, \ldots, x_k)$ be a polynomial in~$\mathbb{F}[x_1, x_2, \ldots, x_k]$. Suppose the degree~$deg(f)$ of~$f$ is $\sum_{i=1}^k \gamma_i$, where each~$\gamma_i$ is a nonnegative integer, and suppose the coefficient of~$\prod_{i=1}^k x_i^{\gamma_i}$ in~$f$ is nonzero. If~$C_1, C_2, \ldots, C_k$ are subsets of~$\mathbb{F}$ with~$|C_i| > \gamma_i$, then there are $c_1 \in C_1, \ldots, c_k \in C_k$ such that~$f(c_1, c_2, \ldots, c_k) \neq 0$.
\end{thm}

In the notation of the Non-Vanishing Corollary, we call the monomial $x_1^{|C_1| - 1}$ $\cdots$ $x_k^{|C_k| - 1}$ the {\em bounding monomial}. The corollary can be rephrased as requiring the polynomial to include a monomial of maximum degree that divides the bounding monomial (where by ``include" we mean that it has a nonzero coefficient).
To use the Non-Vanishing Corollary we require a polynomial for which the nonzeros correspond to successful solutions to the case of the problem under consideration.

For this purpose, we will revisit the procedure of \cite{CDO}. Here, we consider type $(k/2, k/2)$ subsets $S=\{(x_1,a_1),\ldots,(x_k,a_k)\}$ of the group $G=N\rtimes_{\varphi} \Z_2$,  where $N = \Z_p$,  and do not contain $(0,0)$. Under these hypotheses, we obtain a more general sequenceability result than that of \cite{CDO}. Indeed, for a small value of $t$ (i.e. $t\leq 8$), we prove that a subset of $\Z_p\rtimes_{\varphi} \Z_2$ admits an alternating parity $t$-weak sequencing (which will be denoted with $\bm{y}$) regardless of its size $k$.

Since we are looking for an alternating ordering,  we choose\footnote{Note that the ordering $\bm{y}$ could have started with an element of the form $(x_1, 1)$ leading to analogous computations without improving our results.} $\bm{y}$ to be of the form $((x_1,0), (x_2, 1), \ldots,(x_k,1))$.  Hence,  named the second components by $\bm{a} = (a_1, \ldots a_k)$, $a_i$ equals to $(i-1)\pmod{2}$ or, briefly, $(i-1)_2$.  Hence, denoted by $\bm{b} = (b_0, b_1, \ldots, b_k)$ the partial sums of the sequence $\bm{a}$, we must have
$$b_i=\begin{cases} 0 \mbox{ if } i\equiv 0,1 \pmod{4};\\
1 \mbox{ otherwise.}
\end{cases}$$
Also, since the automorphisms of $(\Z_p,+)$ are the multiplications, for each $a \in \Z_2$, we can set $\varphi(a)(x) = \varphi_a x$ for some $\varphi_a \in \Z_p$ where $\varphi_a x$ represents the product of two elements in $\Z_p$. Using these notations the operation of $G$ is defined so that $(x_1,a_1)+_G(x_2,a_2)=(x_1 + \varphi_{a_1} x_2, a_1 + a_2)$.
Now we want to use the polynomial method to show that $S$ admits an alternating parity $t$-weak sequencing.
For this purpose, we consider
$$\bm{y} = \left((x_1, a_1), (x_2, a_2), \ldots, (x_k, a_k) \right)$$
to be a putative arrangement of the elements of~$S$ with partial sums $s_0,s_1, s_2, \ldots, s_k$ given by
\begin{equation*}
(0,0), (x_1, a_1), (x_1, a_1) +_G (x_2, a_2), \ldots, (x_1, a_1) +_G (x_2, a_2) +_G \cdots +_G (x_k, a_k).
\end{equation*}
We note that $s_i = s_j$ if and only if $(x_{i+1},  a_{i+1}) +_G \cdots +_G (x_j, a_j) = 0$ that means $x_{i+1}+ \varphi_{a_{i+1}}x_{i+2}+\varphi_{a_{i+1}+ a_{i+2}}x_{i+3}+\cdots+\varphi_{a_{i+1}+\cdots+ a_{j-1}}x_j = 0$ and $b_j - b_i = a_{i+1} + \cdots + a_j = 0$.

Therefore, given a sequence $\bm{a} = (a_1, \ldots a_k)$ of elements in $\mathbb{Z}_2$ whose partial sums are $\bm{b} = (b_0, b_1, \ldots, b_k)$, we consider, for $t<k$, the following polynomial (see also \cite{CDO}):
\begin{equation}\label{eq: standardpoly} q_{\bm{a}} = \prod_{\substack{1 \leq i < j \leq k \\ j\equiv i\pmod{2} }} (x_j - x_i)
\prod_{\substack{0 \leq i < j \leq k \\ b_i = b_j \\ j-i\leq t \\ j \neq i+1}} (x_{i+1}+ \varphi_{a_{i+1}}x_{i+2}+\varphi_{a_{i+1}+ a_{i+2}}x_{i+3}+\cdots+\varphi_{a_{i+1}+\cdots+ a_{j-1}}x_j).\end{equation}

In this case, we have that a set $S \subseteq G\setminus \{(0,0)\}$ of size $k$ admits an alternating parity $t$-weak sequencing if there exists an ordering $\bm{y} = \left((x_1, 0), (x_2, 1),(x_3, 0), \ldots, (x_k, 1) \right)$ of its elements such that $q_{\bm{a}}(x_1,\ldots, x_k)\not=0$ where $\bm{a} = (0,1, \ldots,0, 1)$.

Now, given a set $S$ of $k$ elements, the idea (following \cite{CD} and \cite{CDO}) is to fix, {\em a priori}, the first $h$ elements $((x_1,0), (x_2,1), \ldots, (x_h, (h-1)_2))$, where $h$ is not too big, of the ordering in such a way that none of its partial sums $s_i, s_j$ are equal when $|i-j|\leq t$ and $1\leq i<j\leq h$. This can be expressed by requiring that $q_{\bm{a}'}(x_1,\ldots,x_h)\not=0$ where $\bm{a}'=(0,\ldots,(h-1)_2)$ and we show that this can be done due the following Proposition whose proof is inspired by that of Proposition 2.2 of \cite{CD} and Proposition 1 of \cite{CDO}.
\begin{prop}\label{fix2}
Let $S \subseteq G\setminus \{(0,0)\}$ be a set of size $k$ and type $(k/2,k/2)$ and let $h$ and $t$ be positive integers such that $k-h\geq (t+2)$ and $h\geq t-1$. Then there exists an alternating parity ordering of $h$-elements of $S$ that we denote, up to relabeling, with $((x_1,0), (x_2,1), \ldots, (x_h, (h-1)_2))$, such that
\begin{itemize}
\item[\normalfont $(*)$] $q_{\bm{a}'}(x_1,\ldots,x_h)\not=0$ where $\bm{a}'=(0,\ldots,(h-1)_2)$;
\end{itemize}
\end{prop}
\begin{proof}
Here we define $S'$ to be the subset of $S$ of the elements in the coset $\Z_p\rtimes_{\varphi}\{0\}$, and we set $S'':=S\setminus S'$.
Now we prove, by induction on $h\leq k-(t+2)$, that we can choose, for any $i \leq h$, $(x_i, a_i) \in S'$ when $i$ is odd and in $S''$ if $i$ is even such that property $(*)$ holds.

Let $h=1$. Since $q_{(a)}(x)=1$ for any $t$ and for any $(x,a) \in S'$, the statement is realized for $h=1$.  

Suppose $h=m+1 \leq k-(t+2)$. Then by induction there exists an $m$-tuple $((x_1, 0),\ldots, (x_m, (m - 1)_2))$ such that
$$q_{(0,1,\ldots, (m-1)_2)}(x_1,\ldots,x_m)\not=0.$$
Given $(x,(m)_2) \in G$, we set $s_m =(x_1, 0) +_G \cdots +_G (x_m, (m-1)_2)=(y_m, b_m)$ and we define $y$, $b$, and $s$ such that $s=s_m +_G (x, (m)_2)=(y, b)$ where $b=b_m+(m)_2$.
Then recalling that $a_i = (i-1)_2$, we have
$$\frac{q_{(0,\ldots, (m-1)_2, (m)_2)}(x_1,\ldots,x_m,x)}{q_{(0,\ldots, (m-1)_2)}(x_1,\ldots,x_m)}=$$
$$\prod_{\substack{1\leq i<m+1\\ i\equiv m+1\pmod{2}}} (x-x_i)\prod_{\substack{0\leq i<m \\ m+1-i\leq t,\ b_i=b}} (x_{i+1}+ \varphi_{a_{i+1}}x_{i+2}+\varphi_{a_{i+1}+ a_{i+2}}x_{i+3}+\cdots+\varphi_{a_{i+1}+\cdots+ a_{m}}x).$$
Here, any element $(x, (m)_2)$ of $S\setminus \{(x_1,0),\ldots, (x_m,(m-1)_2)\}$ satisfies
$\prod_{\substack{1\leq i<m+1\\ i\equiv m+1\pmod{2}}} (x-x_i)\not=0$ since the assumption $m+1\equiv i \pmod{2}$ implies $(m)_2=a_i$ and we have that $(x, (m)_2) \neq (x_i, a_i)$. Hence, to have $\frac{q_{(0,\ldots, (m-1)_2,(m)_2)}(x_1,\ldots,x_m,x)}{q_{(0,\ldots, (m-1)_2)}(x_1,\ldots,x_m)}\not=0$, it suffice to find $x$ such that
$$\prod_{\substack{\max(0,m+1-t)\leq i<m \\ b_i=b}}(x_{i+1}+ \varphi_{a_{i+1}}x_{i+2}+\varphi_{a_{i+1}+ a_{i+2}}x_{i+3}+\cdots+\varphi_{a_{i+1}+\cdots+ a_{m}}x) \neq 0.$$

Note that there is at most one $(x, (m)_2) \in S\setminus\{(x_1, 0),\ldots, (x_m, (m-1)_2)\}$ such that
$$x_{i+1}+ \varphi_{a_{i+1}}x_{i+2}+\varphi_{a_{i+1}+ a_{i+2}}x_{i+3}+\cdots+\varphi_{a_{i+1}+\cdots+ a_{m}}x=0.$$
Moreover, we may have that $b_i=b=b_m+ (m)_2$ only if $i$ and $m+1$, considered modulo $4$ are both either in $\{0,1\}$ or in $\{2,3\}$. If one of these conditions holds, we have that $s_i=s$ (that is $(x_{i+1}, (i)_2) +_G \cdots +_G (x_m, (m-1)_2) +_G (x, (m)_2)=0$) is satisfied by exactly one element $(x, (m)_2)$ of $G$. Note that, for any $m$, we have to consider at most $\lceil t/2\rceil$ relations, according to the congruence classes modulo $4$ of $m$ and $i$. Therefore we have at most $\lceil t/2\rceil$ values $(x, (m)_2)$ in $S\setminus\{(x_1,0),\ldots,(x_m, (m-1)_2)\}$ such that $$\prod_{\substack{\max(0,m+1-t)\leq i<m \\ b_i=b}} (x_{i+1}+ \varphi_{a_{i+1}}x_{i+2}+\varphi_{a_{i+1}+ a_{i+2}}x_{i+3}+\cdots+\varphi_{a_{i+1}+\cdots+ a_{m}}x)=0.$$

Here we have that, if $m+1$ is even, $(x, (m)_2) = (x,1)$ must be chosen in $S''$ and
\begin{multline*}
|S''\setminus \{(x_1,0),\ldots, (x_m, (m-1)_2)\}|= \\ |S''\setminus \{(x_2,1),(x_4, 1),\ldots, (x_{m-1}, 1)\}| \geq k/2-(m-1)/2\geq \lfloor t/2\rfloor+2\,,
\end{multline*}
where the inequality holds since $k-m\geq t+3$.
Since $\lfloor t/2\rfloor+2>\lceil t/2 \rceil$, there exists $(x, (m)_2) \in S''\setminus \{(x_1, 0),(x_2, 1), \ldots, (x_{m-1}, 1), (x_m, 0)\}$ such that
$$\prod_{\substack{\max(0,m+1-t)\leq i<m \\ b_i=b}} (x_{i+1}+ \varphi_{a_{i+1}}x_{i+2}+\varphi_{a_{i+1}+ a_{i+2}}x_{i+3}+\cdots+\varphi_{a_{i+1}+\cdots+ a_{m}}x)\not=0.$$
If instead, $m+1$ is odd, reasoning in a similar way, we have that we can choose $(x, (m)_2) = (x, 0)$ with the same properties in $S'\setminus \{(x_1, 0), (x_2, 1), \ldots, (x_{m-1}, 0), (x_m, 1)\}$.

In both cases, we can find $(x_{m+1}, (m)_2)$ such that
$$\frac{q_{(0,\ldots, (m-1)_2,(m)_2)}(x_1,\ldots,x_m,x_{m+1})}{q_{(0,\ldots, (m-1)_2)}(x_1,\ldots,x_m)}=\prod_{1\leq i<m+1,\ a_{m+1}=a_i} (x_{m+1}-x_i)\cdot$$
$$\cdot\prod_{\substack{\max(0,m+1-t)\leq i<m \\ b_i=b_{m+1}}}(x_{i+1}+ \varphi_{a_{i+1}}x_{i+2}+\varphi_{a_{i+1}+ a_{i+2}}x_{i+3}+\cdots+\varphi_{a_{i+1}+\cdots+ a_{m}}x_{m+1})\not=0.$$
Since $\mathbb{Z}_p$ is a field and due to the inductive hypothesis $q_{(0,\ldots, (m-1)_2)}(x_1,\ldots,x_m) \neq 0$, we also have that
$$\frac{q_{(0,\ldots, (m-1)_2,(m)_2)}(x_1,\ldots,x_m,x_{m+1})}{q_{(0,\ldots, (m-1)_2)}(x_1,\ldots,x_m)}\cdot q_{(0,\ldots, (m-1)_2)}(x_1,\ldots,x_m)=$$
$$=q_{(0,\ldots, (m-1)_2,(m)_2)}(x_1,\ldots,x_{m+1})\not=0 $$
which completes the proof.
\end{proof}
\begin{rem}
Note that in Proposition \ref{fix2} we do not have to assume that $k$ is large enough as we did in \cite{CDO} for Proposition 1. This essentially means that, for balanced sets, we can study the weak sequenceability more effectively using alternating parity sequences.
\end{rem}
In the following we assume that we have fixed $$\{(x_1, 0), (x_2, 1), \ldots, (x_h, (h-1)_2)\}\subseteq S$$ according to Proposition \ref{fix2}. Then we can proceed in simplifying the polynomial $q_{(0,\ldots, (k-1)_2)}$ $(x_1,\ldots,x_h,x_{h+1},\ldots,x_{k})$ following exactly the same steps of \cite{CDO}.
Set $\ell = k-h$ and $z_i=x_{i+h}$, we obtain it is enough to find a non-zero point of the following polynomial
\begin{equation}\label{definizioneH}h_{(0,\ldots, (k-1)_2),\ell}(z,\ldots,z_{\ell}) =\frac{q_{(0,\ldots, (k-1)_2)}(x_1,\ldots,x_h,z_{1},\ldots,z_{\ell})}{q_{(0,\ldots,(h-1)_2)}(x_1,\ldots,x_h)\prod_{1\leq i\leq 1<j\leq \ell,\ a_j=a_i }(z_j-x_i)}.\end{equation}
Since we are also assuming that $h=k-\ell\geq t-1$, that is, $k-(t-1)\geq \ell\geq t+1$, we obtain the following expression
\begin{multline}\label{espressioneH}
h_{(0,\ldots, (k-1)_2),\ell}(z_1,\ldots,z_{\ell})=q_{((h)_2,\ldots, (k-1)_2)}(z_1,\ldots,z_{\ell}) \cdot \\
\prod_{\substack{0\leq i\leq t-1;\\ 1\leq j\leq t-i-1;\\ b_{k-\ell-i-1}=b_{(k-\ell)+j}}} (x_{k-\ell-i}+\varphi_{a_{k-\ell-i}}x_{k-\ell-i+1}+\cdots+\varphi_{a_{k-\ell-i}+\cdots+ a_{k-\ell-1}}x_{k-\ell} + \\ \varphi_{a_{k-\ell-i}+\cdots +a_{k-\ell}}z_1+\cdots+\varphi_{a_{k-\ell-i}+\cdots + a_{k-\ell+j-1}}z_j).
\end{multline}
Now, in order to apply the Non-Vanishing Corollary (of the Combinatorial Nullstellensatz) to the polynomial $h_{(0,\ldots, (k-1)_2),\ell} (z_1,\ldots,z_{\ell})$, it is enough to consider its terms of maximal degree in the variables $z_1,\ldots,z_{\ell}$. We denote by $r_{t,k,\ell}$ the polynomial given by those terms, that is
\begin{multline}\label{espressioneQ2}
r_{t,k,\ell} =q_{((k-\ell)_2,\ldots,(k-1))_2)}(z_1,\ldots,z_{\ell})\cdot \\ \prod_{\substack{0\leq i\leq t-1;\\ 1\leq j\leq t-i-1;\\ b_{k-\ell-i-1}=b_{(k-\ell)+j}}} (\varphi_{a_{k-\ell-i}+\cdots +a_{k-\ell}}z_1+\cdots+\varphi_{a_{k-\ell-i}+\cdots+ a_{k-\ell+j-1}}z_j).
\end{multline}
Since $b_{k-\ell-i-1}=b_{(k-\ell)+j}$ if and only if
$$a_{k-\ell-i} + a_{k-\ell-1}+ \cdots + a_{k-\ell}+ a_{k-\ell+1}+\cdots+ a_{k-\ell+j}= 0$$
where $a_{i}=0$ when $i$ is odd and $a_i=1$ when it is even,
we can state the following, simple but very powerful, remark.
\begin{rem}\label{MagicRemark2}
Since $k$ is even, then the parity of $k-\ell$ does not depend on $k$. It follows that also the expression of $r_{t,k,\ell}$ does not depend $k$ but only on $t$ and $\ell$. Therefore, in the following, we just denote this polynomial by $r_{t,\ell}$.
\end{rem}
Indeed Remark \ref{MagicRemark2} means that, after these manipulations, we are left to consider a polynomial that does not depend on $k=|S|$ and hence we have chances to get a result that is very general on $k$.  We stress again that in this case, contrary to what was done in \cite{CDO}, we do not have to assume that $k$ is large enough.

Then, to apply the Non-Vanishing Corollary, we just need to find a nonzero coefficient in $r_{t,\ell}$ that divides the bounding monomial. In this case, we obtain that:
\begin{thm}\label{th:alternating}
Let $G = \mathbb{Z}_p \rtimes_{\varphi} \mathbb{Z}_2$ with~$p>3$ prime. Then any subset of type $(k/2,k/2)$ of~$G \setminus \{(0,0)\}$ admits an alternating parity $t$-weak sequencing whenever $t\leq \min(8,k-1)$.
\end{thm}
\begin{proof}
Note that a group of size $2p$ is either $\Z_p\times \Z_2$ or the dihedral group $D_{2p}$. Since a $t$-weakly sequenceable subset $S$ of $G \setminus \{(0,0)\}$ is also $(t-1)$-weakly sequenceable, we can suppose that $t = \min(8,k-1)$. Clearly, for $k < 4$ the theorem trivially holds. Therefore we can suppose that $k \geq 4$. We divide the proof considering three different ranges of $k$.

For each $4 \leq k \leq 16$, the polynomial $q_{\bm{a}}$, defined in Equation \eqref{eq: standardpoly} has monomials with non-zero coefficient that divide the bounding monomial $x_1^{k/2-1} x_2^{k/2-1} \cdots x_{k}^{k/2-1}$ (see Tables \ref{tab:111} and \ref{tab:222}). Hence we can apply the Non-Vanishing Corollary so that each subset $S \subseteq G \setminus \{(0,0)\}$, $4 \leq |S| \leq 16$, admits an alternating parity $t$-weak sequencing.

For each $18 \leq k \leq 22$, the polynomial $h_{(0,\ldots, (k-1)_2), \ell}$, defined in eq. \eqref{definizioneH}, for $\ell = 11$ has monomials with non-zero coefficients that divide the bounding monomial $z_1 ^{\ell/2-1} z_2 ^{\ell/2-1}$ $\cdots$ $z_{\ell}^{\ell/2-1}$ (see Tables \ref{tab:11} and \ref{tab:22}). Then thanks to the Non-Vanishing Corollary each subset $S \subseteq G\setminus \{(0,0)\}$, $18 \leq |S| \leq 22$, admits an alternating parity $t$-weak sequencing.

For $k \geq 24$, we consider the polynomials $r_{t, k,\ell}$ defined in eq. \eqref{espressioneQ2} for $\ell \in \{ 16, 17 \}$. Since these polynomials have monomials that divide the bounding monomial with non-zero coefficients (see Tables \ref{tab:1} and \ref{tab:2}), and since $k \geq \ell + t - 1$, we can apply Proposition \ref{fix2} to obtain that each subset $S \subseteq G \setminus \{(0,0)\}$, $|S| \geq 24$, admits an alternating parity $t$-weak sequencing.
\end{proof}

\small
\renewcommand*{\arraystretch}{1.4}
\begin{longtable}{lllll}
\caption{Monomials sufficient for the proof of Theorem~\ref{th:alternating} in the case $t=\min(8,k-1)$, $4 \leq |S| \leq 16$ and~$G=\mathbb{Z}_p \times \mathbb{Z}_2$.}\\
\hline
$k$ & deg & monomial/s & coefficient/s \\
\hline
$4$ & $3$ & $x_1 x_2 x_3$ & $1$\\
\hline
$6$ & $12$ & $x_1^2 x_2^2 x_3^2 x_4^2 x_5^2 x_6^2$ & $-2^2$\\
\hline
$8$ & $23$ & $x_1^3 x_2^2 x_3^3 x_4^3 x_5^3 x_6^3 x_7^3 x_8^3$ & $ 2^3 \cdot 3^2$\\
\hline
$10$ & $39$ & \begin{tabular}{@{}l@{}} $x_1^3 x_2^4 x_3^4 x_4^4 x_5^4 x_6^4 x_7^4 x_8^4 x_9^4 x_{10}^4$ \\ $x_1^4 x_2^3 x_3^4 x_4^4 x_5^4 x_6^4 x_7^4 x_8^4 x_9^4 x_{10}^4$ \end{tabular} & \begin{tabular}{@{}l@{}} $-112223$ \\ $17 \cdot 5051$ \end{tabular}\\
\hline
$12$ & $56$ & \begin{tabular}{@{}l@{}} $x_1^3 x_2^3 x_3^5 x_4^5 x_5^5 x_6^5 x_7^5 x_8^5 x_9^5 x_{10}^5 x_{11}^5 x_{12}^5$ \\ $x_1^2 x_2^4 x_3^5 x_4^5 x_5^5 x_6^5 x_7^5 x_8^5 x_9^5 x_{10}^5 x_{11}^5 x_{12}^5$ \end{tabular} & \begin{tabular}{@{}l@{}} $2 \cdot 1738728373$ \\ $2^2 \cdot 7 \cdot 98908039$ \end{tabular}\\
\hline
$14$ & $75$ & \begin{tabular}{@{}l@{}} $x_1^2 x_2^3 x_3^4 x_4^6 x_5^6 x_6^6 x_7^6 x_8^6 x_9^6 x_{10}^6 x_{11}^6 x_{12}^6 x_{13}^6 x_{14}^6$ \\ $x_1^3 x_2^3 x_3^3 x_4^6 x_5^6 x_6^6 x_7^6 x_8^6 x_9^6 x_{10}^6 x_{11}^6 x_{12}^6 x_{13}^6 x_{14}^6$ \end{tabular} & \begin{tabular}{@{}l@{}} $251 \cdot 4787 \cdot 11018171$ \\ $3 \cdot 8429305776581$ \end{tabular}\\
\hline
$16$ & $96$ & \begin{tabular}{@{}l@{}} $x_1^2 x_2^2 x_3^2 x_4^6 x_5^7 x_6^7 x_7^7 x_8^7 x_9^7 x_{10}^7 x_{11}^7 x_{12}^7 x_{13}^7 x_{14}^7 x_{15}^7 x_{16}^7$ \\ $x_1^2 x_2^2 x_3^3 x_4^5 x_5^7 x_6^7 x_7^7 x_8^7 x_9^7 x_{10}^7 x_{11}^7 x_{12}^7 x_{13}^7 x_{14}^7 x_{15}^7 x_{16}^7$ \end{tabular} & \begin{tabular}{@{}l@{}} $-2^4 \cdot 229 \cdot 27900069285799$ \\ $-2^2 \cdot 3 \cdot 5 \cdot 70141 \cdot 29906598707$ \end{tabular}\\
\hline\label{tab:111}
\end{longtable}

\small
\renewcommand*{\arraystretch}{1.4}
\begin{longtable}{lllll}
\caption{Monomials sufficient for the proof of Theorem~\ref{th:alternating} in the case $t=\min(8,k-1)$, $4 \leq |S| \leq 16$ and~$G=D_{2p}$.}\\
\hline
$k$ & deg & monomial/s & coefficient/s \\
\hline
$4$ & $3$ & $x_1 x_2 x_3$ & $-1$\\
\hline
$6$ & $12$ & $x_1^2 x_2^2 x_3^2 x_4^2 x_5^2 x_6^2$ & $-2^2 \cdot 3 \cdot 13$\\
\hline
$8$ & $23$ & \begin{tabular}{@{}l@{}} $x_1^2 x_2^3 x_3^3 x_4^3 x_5^3 x_6^3 x_7^3 x_8^3$ \\ $x_1^3 x_2^2 x_3^3 x_4^3 x_5^3 x_6^3 x_7^3 x_8^3$ \end{tabular} & \begin{tabular}{@{}l@{}} $2^2 \cdot 3 \cdot 1087$ \\ $ 2^3 \cdot 2029$ \end{tabular}\\
\hline
$10$ & $39$ & \begin{tabular}{@{}l@{}} $x_1^3 x_2^4 x_3^4 x_4^4 x_5^4 x_6^4 x_7^4 x_8^4 x_9^4 x_{10}^4$ \\ $x_1^4 x_2^3 x_3^4 x_4^4 x_5^4 x_6^4 x_7^4 x_8^4 x_9^4 x_{10}^4$ \end{tabular} & \begin{tabular}{@{}l@{}} $3^2 \cdot 5 \cdot 71 \cdot 257 \cdot 421$ \\ $5^2 \cdot 7^2 \cdot 547 \cdot 1087$ \end{tabular}\\
\hline
$12$ & $56$ & \begin{tabular}{@{}l@{}} $x_1^3 x_2^3 x_3^5 x_4^5 x_5^5 x_6^5 x_7^5 x_8^5 x_9^5 x_{10}^5 x_{11}^5 x_{12}^5$ \\ $x_1^2 x_2^4 x_3^5 x_4^5 x_5^5 x_6^5 x_7^5 x_8^5 x_9^5 x_{10}^5 x_{11}^5 x_{12}^5$ \end{tabular} & \begin{tabular}{@{}l@{}} $2 \cdot 687920864801$ \\ $ 2 \cdot 3^3 \cdot 353 \cdot 102554917$ \end{tabular}\\
\hline
$14$ & $75$ & \begin{tabular}{@{}l@{}} $x_1^2 x_2^3 x_3^4 x_4^6 x_5^6 x_6^6 x_7^6 x_8^6 x_9^6 x_{10}^6 x_{11}^6 x_{12}^6 x_{13}^6 x_{14}^6$ \\ $x_1^3 x_2^3 x_3^3 x_4^6 x_5^6 x_6^6 x_7^6 x_8^6 x_9^6 x_{10}^6 x_{11}^6 x_{12}^6 x_{13}^6 x_{14}^6$ \end{tabular} & \begin{tabular}{@{}l@{}} $5 \cdot 7 \cdot 11 \cdot 13 \cdot 19 \cdot 439 \cdot 1033 \cdot 545257$ \\ $-3 \cdot 5 \cdot 7 \cdot 160690844448907$ \end{tabular}\\
\hline
$16$ & $96$ & \begin{tabular}{@{}l@{}} $x_1^2 x_2^2 x_3^2 x_4^6 x_5^7 x_6^7 x_7^7 x_8^7 x_9^7 x_{10}^7 x_{11}^7 x_{12}^7 x_{13}^7 x_{14}^7 x_{15}^7 x_{16}^7$ \\ $x_1^2 x_2^2 x_3^3 x_4^5 x_5^7 x_6^7 x_7^7 x_8^7 x_9^7 x_{10}^7 x_{11}^7 x_{12}^7 x_{13}^7 x_{14}^7 x_{15}^7 x_{16}^7$ \end{tabular} & \begin{tabular}{@{}l@{}} $-2^5 \cdot 10567 \cdot 7458067 \cdot 21508099$ \\ $-2^2 \cdot 3 \cdot 7 \cdot 19 \cdot 151 \cdot 331 \cdot 5657 \cdot 43978321$ \end{tabular}\\
\hline\label{tab:222}
\end{longtable}
\vspace{-0.1cm}
\small
\renewcommand*{\arraystretch}{1.2}
\begin{longtable}{lllll}
\caption{Monomials sufficient for the proof of Theorem~\ref{th:alternating} in the case $t=8$, $18 \leq |S| \leq 22$ and~$G=\mathbb{Z}_p \times \mathbb{Z}_2$.}\\
\hline
$k$ & $\ell$ & deg & monomial/s & coefficient/s \\
\hline
$22$ & $16$ & $111$ & \begin{tabular}{@{}l@{}} $z_1^{6} z_2^{7} z_3^{7} z_4^{7} z_5^{7} z_6^{7} z_7^{7} z_8^{7} z_9^{7} z_ {10}^{7} z_{11}^{7} z_{12}^{7} z_{13}^{7} z_{14}^{7} z_ {15}^{7} z_{16}^{7}$ \\ $z_1^{7} z_2^{6} z_3^{7} z_4^{7} z_5^{7} z_6^{7} z_7^{7} z_8^{7} z_9^{7} z_ {10}^{7} z_{11}^{7} z_{12}^{7} z_{13}^{7} z_{14}^{7} z_ {15}^{7} z_{16}^{7}$ \end{tabular} & \begin{tabular}{@{}l@{}} $-2^{10} \cdot 5137080631 \cdot 34602352027$ \\ $-19 \cdot 181 \cdot 1151 \cdot 49118739946632313$ \end{tabular} \\
\hline
$20$ & $16$ & $108$ & \begin{tabular}{@{}l@{}} $z_1^{3} z_2^{7} z_3^{7} z_4^{7} z_5^{7} z_6^{7} z_7^{7} z_8^{7} z_9^{7} z_ {10}^{7} z_{11}^{7} z_{12}^{7} z_{13}^{7} z_{14}^{7} z_ {15}^{7} z_{16}^{7}$ \\ $z_1^{4} z_2^{6} z_3^{7} z_4^{7} z_5^{7} z_6^{7} z_7^{7} z_8^{7} z_9^{7} z_ {10}^{7} z_{11}^{7} z_{12}^{7} z_{13}^{7} z_{14}^{7} z_ {15}^{7} z_{16}^{7}$ \end{tabular} & \begin{tabular}{@{}l@{}} $-5 \cdot 69707639 \cdot 13329296986877$ \\ $-2^2 \cdot 313 \cdot 2543 \cdot 419171 \cdot 9848994079$ \end{tabular} \\
\hline
$18$ & $16$ & $103$ & \begin{tabular}{@{}l@{}} $z_1^{2} z_2^{3} z_3^{7} z_4^{7} z_5^{7} z_6^{7} z_7^{7} z_8^{7} z_9^{7} z_ {10}^{7} z_{11}^{7} z_{12}^{7} z_{13}^{7} z_{14}^{7} z_ {15}^{7} z_{16}^{7}$ \\ $z_1^{2} z_2^{4} z_3^{6} z_4^{7} z_5^{7} z_6^{7} z_7^{7} z_8^{7} z_9^{7} z_ {10}^{7} z_{11}^{7} z_{12}^{7} z_{13}^{7} z_{14}^{7} z_ {15}^{7} z_{16}^{7}$ \end{tabular} & \begin{tabular}{@{}l@{}} $-2 \cdot 911 \cdot 1877 \cdot 691183 \cdot 98867581$ \\ $-2^4 \cdot 7 \cdot 227 \cdot 419 \cdot 1249 \cdot 46499 \cdot 556267$ \end{tabular} \\
\hline
\label{tab:11}
\end{longtable}

\small
\begin{longtable}{lllll}
\caption{Monomials sufficient for the proof of Theorem~\ref{th:alternating} in the case $t=8$, $18 \leq |S| \leq 22$ and~$G=D_{2p}$.}\\
\hline
$k$ & $\ell$ & deg & monomial/s & coefficient/s \\
\hline
$22$ & $16$ & $111$ & \begin{tabular}{@{}l@{}} $z_1^{6} z_2^{7} z_3^{7} z_4^{7} z_5^{7} z_6^{7} z_7^{7} z_8^{7} z_9^{7} z_ {10}^{7} z_{11}^{7} z_{12}^{7} z_{13}^{7} z_{14}^{7} z_ {15}^{7} z_{16}^{7}$ \vspace{0.2cm} \\ $z_1^{7} z_2^{6} z_3^{7} z_4^{7} z_5^{7} z_6^{7} z_7^{7} z_8^{7} z_9^{7} z_ {10}^{7} z_{11}^{7} z_{12}^{7} z_{13}^{7} z_{14}^{7} z_ {15}^{7} z_{16}^{7}$ \end{tabular} & \begin{tabular}{@{}l@{}} $2 \cdot 19 \cdot 73 \cdot 1753 \cdot 1913 \cdot 2927 \cdot 10709$ \\ $16561 \cdot 142433$ \\ $257 \cdot 614556821477 \cdot 4935889852087$ \end{tabular} \\
\hline
$20$ & $16$ & $108$ & \begin{tabular}{@{}l@{}} $z_1^{3} z_2^{7} z_3^{7} z_4^{7} z_5^{7} z_6^{7} z_7^{7} z_8^{7} z_9^{7} z_ {10}^{7} z_{11}^{7} z_{12}^{7} z_{13}^{7} z_{14}^{7} z_ {15}^{7} z_{16}^{7}$ \\ $z_1^{4} z_2^{6} z_3^{7} z_4^{7} z_5^{7} z_6^{7} z_7^{7} z_8^{7} z_9^{7} z_ {10}^{7} z_{11}^{7} z_{12}^{7} z_{13}^{7} z_{14}^{7} z_ {15}^{7} z_{16}^{7}$ \end{tabular} & \begin{tabular}{@{}l@{}} $-20280066499283581979170979$ \\ $-2^4 \cdot 3^3 \cdot 7 \cdot 48091 \cdot 27400883 \cdot 12388512137$ \end{tabular} \\
\hline
$18$ & $16$ & $103$ & \begin{tabular}{@{}l@{}} $z_1^{2} z_2^{3} z_3^{7} z_4^{7} z_5^{7} z_6^{7} z_7^{7} z_8^{7} z_9^{7} z_ {10}^{7} z_{11}^{7} z_{12}^{7} z_{13}^{7} z_{14}^{7} z_ {15}^{7} z_{16}^{7}$ \\ $z_1^{2} z_2^{4} z_3^{6} z_4^{7} z_5^{7} z_6^{7} z_7^{7} z_8^{7} z_9^{7} z_ {10}^{7} z_{11}^{7} z_{12}^{7} z_{13}^{7} z_{14}^{7} z_ {15}^{7} z_{16}^{7}$ \end{tabular} & \begin{tabular}{@{}l@{}} $-2 \cdot 11 \cdot 12541 \cdot 2479011358328177663$ \\ $2 \cdot 3^2 \cdot 113 \cdot 167 \cdot 1953595178092313773$ \end{tabular} \\
\hline
\label{tab:22}
\end{longtable}

\small
\begin{longtable}{llll}
\caption{Monomials for the proof of Theorem \ref{th:alternating} in the case $t=8$, $|S| \geq 24$ and~$G=\mathbb{Z}_p \times \mathbb{Z}_2$.}\label{tab:1}\\
\hline
$\ell$ & deg & monomial/s & coefficient/s \\
\hline
\endhead
$16$ & $112$ & $z_1^{7} z_2^{7} z_3^{7} z_4^{7} z_5^{7} z_6^{7} z_7^{7} z_8^{7} z_9^{7} z_ {10}^{7} z_{11}^{7} z_{12}^{7} z_{13}^{7} z_{14}^{7} z_ {15}^{7} z_{16}^{7}$ & $23 \cdot 964303 \cdot 2069134328306807$\\
\hline
$17$ & $124$ & $z_1^{5} z_2^{6} z_3^{8} z_4^{7} z_5^{8} z_6^{7} z_7^{8} z_8^{7} z_9^{8} z_ {10}^{7} z_{11}^{8} z_{12}^{7} z_{13}^{8} z_{14}^{7} z_ {15}^{8} z_{16}^{7} z_{17}^{8} $ & $-19 \cdot 25169 \cdot 2392881694552215689$\\
\hline
\end{longtable}

\small
\begin{longtable}{llll}
\caption{Monomials for the proof of Theorem \ref{th:alternating} in the case $t=8$, $|S| \geq 24$ and~$G=D_{2p}$.}\label{tab:2}\\
\hline
$\ell$ & deg & monomial/s & coefficient/s \\
\hline
\endhead
$16$ & $112$ & $z_1^{7} z_2^{7} z_3^{7} z_4^{7} z_5^{7} z_6^{7} z_7^{7} z_8^{7} z_9^{7} z_ {10}^{7} z_{11}^{7} z_{12}^{7} z_{13}^{7} z_{14}^{7} z_ {15}^{7} z_{16}^{7}$ & $-46457 \cdot 6420261196031028971393$\\
\hline
$17$ & $124$ & $z_1^{5} z_2^{6} z_3^{8} z_4^{7} z_5^{8} z_6^{7} z_7^{8} z_8^{7} z_9^{8} z_ {10}^{7} z_{11}^{8} z_{12}^{7} z_{13}^{8} z_{14}^{7} z_ {15}^{8} z_{16}^{7} z_{17}^{8} $ & $-3 \cdot 5 \cdot 2699235801551 \cdot 113276824810481$\\
\hline
\end{longtable}
\normalsize
\section{A Ramsey/Probabilistic Approach}
The goal of this section is to prove Theorems \ref{th:main3} and \ref{th:main1}. For this purpose, we will introduce a hybrid Ramsey theoretical and Probabilistic approach and we will apply it, firstly in the case of standard weak sequenceability, and then in the case of balanced set to obtain alternating parity weak sequencings.

First of all, we recall the following Theorem (see \cite{EHR}):
\begin{thm}\label{HappyTheorem}
Given a positive integer $t$, there exists $R^t(l,h)\in \mathbb{N}$ such that, for any $n\geq R^t(l,h)$ and any coloring $c: {[1,\ldots,n]\choose t}\rightarrow \{\mbox{red, black}\}$, there exists a set $T\in {[1,\ldots,n]\choose l}$ such that $c(Z)=\mbox{ red}$ for every $Z \in {T\choose t}$ or there exists $T\in {[1,\ldots,n]\choose h}$ such that $c(Z)=\mbox{ black}$ for every $Z\in {T\choose t}$.
\end{thm}
Then we apply this theorem to prove the following Ramsey-like proposition:
\begin{prop}\label{Ramsey}
Let $S$ be a subset of size $k$ of a \textup{(}not necessarily abelian\textup{)} group $(G, +)$ and let $t$, $\ell$ be positive integers. Then, there exists a constant $k_{t,\ell}$ such that, if $k>k_{t,\ell}$, $S$ contains a subset $T$ whose size is at least $\ell$ that does not admit zero-sum subsets \textup{(}with respect to any ordering\textup{)} of size $h\leq t$.
\end{prop}
\proof
We proceed by induction on $t$.

BASE CASE, $t=1$. Here we have that, if $|S|\geq \ell+1$, then $S'=S\setminus \{0\}$ has size at least $\ell$. Since $S'$ does not contain $0$, any subset of size $1$ of $S'$ does not sum to zero.

INDUCTIVE STEP. Due to the inductive hypothesis, if the size of $S$ is bigger than $k_{t-1,R^t(t!+t,\ell)}$, $S$ contains a subset $S'$ whose size is at least $R^t(t!+t,\ell)$ and that does not admit zero-sum subset (with respect to any ordering) of size $h\leq t-1$.

Now we consider the $t$-uniform hypergraph given by the $t$-subsets of $S'$ (i.e. the elements of ${S' \choose t}$). Given $Z=\{y_1,\ldots,y_{t}\}\in {S' \choose t}$, we color $Z$ in red if it admits an ordering $\sigma$ such that $y_{\sigma(1)}+y_{\sigma(2)}+\cdots + y_{\sigma(t)}=0$ and in black otherwise. Due to Theorem \ref{HappyTheorem}, since $|S'|\geq R^t(t!+t,\ell)$, $S'$ has either a subset $T$ of size $t!+t$ whose subsets of size $t$ all sum to zero with respect to some ordering (i.e. are red) or a subset $T$ of size $\ell$ whose subsets of size $t$ have all non-zero-sum (i.e. are black).
Let us assume we are in the first case and let us consider a subset $Z'=\{y_1,\ldots,y_{t-1}\}$ of $T$ of size $t-1$. We note that, for any permutation of $Z'$, and for any position $h$ of $x$ the following equation admits only one solution $x\in G$:
$$y_{\sigma(1)}+\cdots +y_{\sigma(h-1)}+x+y_{\sigma(h)}+\cdots+y_{\sigma(t-1)}=0.$$
It follows that there are at most $t!$ ways to extend $Z'$ to a zero-sum set $B$ of size $t$. Since $|T|=|Z'|+t!+1 = t! + t$, $T$ contains $t!+1$ extensions of $Z'$ to subsets of size $t$, and one of these extensions must have a non-zero-sum (with respect to any ordering). But this contradicts the assumption that we are in the first case.

Therefore, we have that $S'$ contains a subset $T$ of size $\ell$ whose subsets of size $t$ have all non-zero-sum (with respect to any ordering). Because of the inductive hypothesis, also all the subsets of $T$ of size $h\leq t-1$ have non-zero sum and hence the thesis is verified.
\endproof
We recall that, in \cite{CD}, it was stated\footnote{Actually, in \cite{CD},  the theorem was stated only in case $G=\mathbb{Z}_n$ but Proposition \ref{fix3} can be obtained with the same proof.} that:
\begin{prop}\label{fix3}
Let $S=\{y_1,\ldots,y_k\}\subseteq G\setminus\{0\}$ be a set of size $k$ and let $h$ and $t$ be positive integers such that $h\leq k-(t-1)$. Then there is an ordering of $h$-elements of $S$ that we denote, up to relabeling, with $(y_1,\ldots,y_h)$, such that, for any ${0\leq i<j\leq \min(h,i+t)}$
$$s_i=y_1+y_2+\cdots+y_i\not=y_1+y_2+\cdots+y_j=s_j.$$\end{prop}
Then, using a probabilistic approach, we are able to prove Theorem \ref{th:main1}:
\begin{thm*2}
Let $(G,  +)$ be a group of size $n$. Then subsets of size~$k$ of~$G\setminus \{0\}$ are $t$-weakly sequenceable whenever $k$ is large enough with respect to $t$.
\end{thm*2}
\proof
Let $S$ be a subset of size $k$ of $G\setminus \{0\}$. According to Proposition \ref{Ramsey}, however we choose $\ell$, if we assume that $k$ is large enough, there exists a subset $T$ of $S$ whose size is $\ell$ and that does not admit zero-sum subset of size $h\leq t$. Let $U=S\setminus T$ and let $k'=k-\ell$ be the size of $U$. According to Proposition \ref{fix3}, given $h= k'-(t-1)=k-\ell-(t-1)$, we can order $h$ elements of $U$, namely $(y_1,\ldots,y_h)$, in such a way that, for any ${0\leq i<j\leq \min(h,i+t)}$
$$s_i=y_1+y_2+\cdots+y_i\not=y_1+y_2+\cdots+y_j=s_j.$$
We denote by $U'$ the set $\{y_1,\ldots,y_h\}$.
Now it suffices to order the remaining $t-1$ elements of $U$ and the $\ell$ elements of $T$. Here we choose, uniformly at random an ordering $y_1,y_2,\ldots,y_h,z_{h+1},\ldots, z_{k}$ that extend the one of $U'$. Let $X$ be the random variable that represents the number of pairs $(i, j)$ such that $s_i = s_j$ with $0\leq i<j \leq k$ and $j - i \leq t$. We evaluate the expected value of $X$.
Because of the linearity of the expectation, we have that:
\begin{equation}\label{E1}\mathbb{E}(X)=\sum_{\substack{0\leq i<j \leq h\\j - i \leq t}} \mathbb{P}(s_i=s_j)+\sum_{\substack{0\leq i\leq h<j\leq k\\j - i \leq t}} \mathbb{P}(s_i=s_j)+\sum_{\substack{h< i<j \leq k\\j - i \leq t}} \mathbb{P}(s_i=s_j).\end{equation}
Due to the choice of $U'$, when $0\leq i<j \leq h$, the probability $\mathbb{P}(s_i=s_j)=0$. Assuming that $0\leq i\leq h<j\leq k$,
since the equation
$$s_i=y_1+y_2+\cdots+y_i =y_1+y_2+\cdots+y_{h}+z_{h+1}+\cdots+z_{j-1}+x=s_j$$ admits only one solution, $\mathbb{P}(s_i=s_j)$ is smaller than or equal to $\frac{1}{|S\setminus (U'\cup\{z_{h+1},\ldots,z_{j-1}\})|}\leq 1/\ell.$ Here the last inequality holds because $j$ is at most $h+t$ and hence the cardinality of $|S\setminus (U'\cup\{z_{h+1},\ldots,z_{j-1}\})|$ is at least $k-h-(t-1)=\ell$.
It follows that Equation \eqref{E1} can be rewritten as
\begin{equation}\label{E2}\mathbb{E}(X)\leq 0+\frac{t^2}{\ell}+\sum_{\substack{h< i<j \leq k\\j - i \leq t}} \mathbb{P}(s_i=s_j).\end{equation}
It is easy to see that
$$\sum_{\substack{h< i<j \leq k\\j - i \leq t}} \mathbb{P}(s_i=s_j)=\sum_{\substack{h< i<j \leq k\\j - i \leq t}} \mathbb{P}(z_{i+1}+z_{i+2}+\cdots+z_j=0).$$
We divide the evaluation of this probability in two cases according to whether $\{z_{i+1},$ $z_{i+2}, \ldots, z_j\}\subseteq T$ or not. In the second case, since at least one of the elements $z_{i+1},$ $z_{i+2},\ldots,z_j$ is not in $T$, we can assume that $z_{i+1}\not \in T$ and since $z_{i+1}\not=0$ we can assume $j>i+1$. Reasoning as before we obtain that, under these assumptions, the probability of $z_{i+1}+z_{i+2}+\cdots+z_j=0$ is at most $\frac{1}{\ell}$.
Therefore we have that:
$$\mathbb{P}(z_{i+1}+z_{i+2}+\cdots+z_j=0)\leq \frac{1}{\ell}\mathbb{P}(\{z_{i+1},z_{i+2},\cdots,z_j\}\not\subseteq T)+$$ $$\mathbb{P}(z_{i+1},z_{i+2},\cdots,z_j\in T)\mathbb{P}(z_{i+1}+z_{i+2}+\cdots+z_j=0|z_{i+1},z_{i+2},\cdots,z_j\in T)$$
and, since $z_{i+1}+z_{i+2}+\cdots+z_j\not=0$ whenever $z_{i+1},z_{i+2},\cdots,z_j\in T$,
$$\mathbb{P}(z_{i+1}+z_{i+2}+\cdots+z_j=0)\leq \frac{1}{\ell}\mathbb{P}(\{z_{i+1},z_{i+2},\cdots,z_j\}\not\subseteq T).$$
Since $z_{i+1},z_{i+2},\ldots,z_j$ are at most $t$ elements randomly chosen in $S\setminus U'$, the probability that all of them are contained in $T$ is at least $$\frac{{|T|\choose j-i}}{{|S\setminus U'|\choose j-i}}\geq\frac{{|T|\choose t}}{{|S\setminus U'|\choose t}} = \frac{{\ell \choose t}}{{\ell+t\choose t}}\geq \frac{(\ell-(t-1))^{t}}{(\ell+t)^{t}}.$$
It follows that
$$\mathbb{P}(z_{i+1}+z_{i+2}+\cdots+z_j=0)\leq \frac{1}{\ell}\left(1-\frac{(\ell-(t-1))^{t}}{(\ell+t)^{t}}\right)$$
and hence Equation \eqref{E2} becomes
\begin{equation}\label{E3}\mathbb{E}(X)\leq \frac{t^2}{\ell}+\frac{1}{\ell}\sum_{\substack{h< i<j \leq k\\j - i \leq t}} \left(1-\frac{(\ell-(t-1))^{t}}{(\ell+t)^{t}}\right).\end{equation}
Finally, since there are less than $(\ell+t)t$ pairs $(i,j)$ such that $h< i<j \leq k$ and $j - i \leq t$, from Equation \eqref{E3} we obtain
\begin{equation}\label{E4}\mathbb{E}(X)\leq \frac{t^2}{\ell}+\frac{t(\ell+t)}{\ell}\left(1-\frac{(\ell-(t-1))^{t}}{(\ell+t)^{t}}\right).\end{equation}
Here we note that the right-hand side of the last inequality goes to zero as $\ell$ goes to infinite. Therefore for $\ell$ large enough, or more precisely for $\ell\geq\bar{\ell}$ for a suitable $\bar{\ell}\in \mathbb{N}$, we have that $\mathbb{E}(X)<1$. This means that, if $k$ is large enough (i.e. if $k>k_{t,\bar{\ell}}$) there exists an ordering on which $X=0$ that is, there exists a $t$-weak sequencing of $S$.
\endproof
\subsection{Revisiting the Probabilistic Approach}
In this section, we will adapt the previous approach to obtain alternating parity $t$-weak sequencings.
First of all, we need to prove a modified version of the Ramsey theoretical Proposition \ref{Ramsey}:
\begin{prop}\label{Ramsey2}
Let $S$ be a subset of type $(k/2,k/2)$ of the group $G=N\rtimes_{\varphi}\Z_2$. Then, given an even value of $\ell$, there exists a constant $\bar{k}_{t,\ell}$ such that, if $k>\bar{k}_{t,\ell}$, $S$ contains a subset $T$ of type $(\ell/2,\ell/2)$ that does not admit subsets of size $h\leq t$ with alternating parity zero-sum sequencings.
\end{prop}
\proof
Now we apply an inductive idea similar to that of Proposition \ref{Ramsey}.
Here we can assume that $S$ contains a subset $S'$ of type $(k'/2,k'/2)$, where $k'$ is arbitrarily large, that does not admit subsets of size $h< t$ with alternating parity zero-sum sequencings.

Then, we note that, if $t\equiv 2 \pmod{4}$, any balanced set of size $t$ does not sum to zero (in any ordering). Otherwise we set $$ t'=\begin{cases}
t/2  & \text{if } t\equiv 0 \pmod{4};\\
\lfloor t/2 \rfloor & \text{if } t\equiv 1 \pmod{4};\\
\lceil t/2\rceil & \text{if } t\equiv 3 \pmod{4}.
\end{cases}
$$
Now we fix a subset $T'$ of $S'$ of type $(\ell/2,0)$ and we want to apply Theorem \ref{HappyTheorem} to the extensions of $T'$ with elements of $S'\cap (N\rtimes_{\varphi}\{1\})$. We note that a subset of $S'$ of size $t$ admits an alternating parity zero-sum sequencing only if it has $t'$ elements in $S' \cap (N\rtimes_{\varphi}\{1\})$ and $t-t'$ in $S' \cap (N\rtimes_{\varphi}\{0\})$.

Then we color a subset $Z\subseteq S'\cap (N\rtimes_{\varphi}\{1\})$ of size $t'$ in red if there exists $Z'\subseteq T'$ whose size is $t-t'$ such that the sum of the elements of $Z$ and $Z'$ is zero in some alternating parity order. Otherwise, we color $Z$ in black. Following the reasoning of Proposition \ref{Ramsey}, if $k'/2$ is large enough, we end up with a black, monochromatic complete $t'$-uniform hypergraph of size $\ell/2$.

It follows that there exists a constant $\bar{k}_{t,\ell}$ such that, if $k>\bar{k}_{t,\ell}$, $S$ contains a subset $T$ of type $(\ell/2,\ell/2)$ that does not admit subsets of size $h\leq t$ with alternating parity zero-sum sequencings.
\endproof
We also recall that with essentially the same proof of Proposition \ref{fix2} we obtain that
\begin{prop}\label{fix5}
Let $S \subseteq G\setminus \{(0,0)\}$ be a set of size $k$ and type $(k/2,k/2)$ and let $h$ and $t$ be positive integers such that $k-h\geq t+2$ and $h\geq t-1$. Then there is an alternating parity ordering of $h$-elements of $S$ that we denote with $(y_1, y_2, \ldots, y_h)$, such that, for any ${0\leq i<j\leq \min(h,i+t)}$
$$s_i=y_1+y_2+\cdots+y_i\not=y_1+y_2+\cdots+y_j=s_j.$$
\end{prop}

Then, following the line of Theorem \ref{th:main1}, we can prove that
\begin{thm*}
Let $G=N\rtimes_{\varphi} \Z_2$ be a group of size $n$. Then subsets of type $(k/2,k/2)$ of~$G\setminus \{(0,0)\}$ admit an alternating parity $t$-weak sequencing whenever $k$ is large enough with respect to $t$.
\end{thm*}
\proof
In the following, we will assume, without loss of generality, that $t$ is even. Indeed, if the thesis is satisfied for a given value of $t$ then it holds also for smaller ones.
Let $S$ be a subset of type $(k/2,k/2)$ of $G\setminus \{(0,0)\}$. According to Proposition \ref{Ramsey2}, however we choose an even $\ell$, if we assume that $k$ is large enough, there exists a subset $T$ of $S$ of type $(\ell/2,\ell/2)$ that does not admit subsets of size $h\leq t$ with alternating parity zero-sum sequencings.

Let us set $U=S\setminus T$ and let $k'=k-\ell$ be the size of $U$. According to Proposition \ref{fix5}, given $h=k'-(t+2)=k-\ell-(t+2)$, we can provide an alternating parity sequencing to $h$ elements of $U$, namely $(y_1,\ldots,y_h)$, in such a way that, for any ${0\leq i<j\leq \min(h,i+t)}$
$$s_i=y_1+y_2+\cdots+y_i\not=y_1+y_2+\cdots+y_j=s_j.$$
We denote by $U'$ the set $\{y_1,\ldots,y_h\}$.
Now it suffices to order the remaining $t+2$ elements of $U$ and the $\ell$ elements of $T$. Here we choose, uniformly at random an alternating parity ordering $y_1,y_2,\ldots,y_h,$ $z_{h+1},\ldots, z_{k}$ that extends the one of $U'$. Let $X$ be the random variable that represents the number of pairs $(i, j)$ such that $s_i = s_j$ with $0\leq i<j \leq k$ and $j - i \leq t$. We evaluate the expected value of $X$.
Because of the linearity of the expectation, we have that:
\begin{equation}\label{EB1}\mathbb{E}(X)=\sum_{\substack{0\leq i<j \leq h\\j - i \leq t}} \mathbb{P}(s_i=s_j)+\sum_{\substack{0\leq i\leq h<j\leq k\\j - i \leq t}} \mathbb{P}(s_i=s_j)+\sum_{\substack{h< i<j \leq k\\j - i \leq t}} \mathbb{P}(s_i=s_j).\end{equation}
Due to the choice of $U'$, when $0\leq i<j \leq h$, the probability $\mathbb{P}(s_i=s_j)=0$. Assuming that $0\leq i\leq h<j\leq k$,
since the equation
$$s_i=y_1+y_2+\cdots+y_i = y_1+y_2+\cdots+y_{h}+z_{h+1}+\cdots+z_{j-1}+x=s_j$$ admits only one solution and the coset of $x$ is $N\rtimes_{\varphi}{(j-1)_2}$, $\mathbb{P}(s_i=s_j)$ is smaller than or equal to $\frac{2}{|S\setminus (U'\cup\{z_{h+1},\ldots, z_{j-1}\}) |-1} \leq \frac{2}{\ell+2}$ because $j$ is at most $h+t$. It follows that Equation \eqref{EB1} can be written as
\begin{equation}\label{EB2}\mathbb{E}(X)\leq 0+\frac{2 t^2}{\ell+2}+\sum_{\substack{h< i<j \leq k\\j - i \leq t}} \mathbb{P}(s_i=s_j).\end{equation}
It is easy to see that
$$\sum_{\substack{h< i<j \leq k\\j - i \leq t}} \mathbb{P}(s_i=s_j)=\sum_{\substack{h< i<j \leq k\\j - i \leq t}} \mathbb{P}(z_{i+1}+z_{i+2}+\cdots+z_j=0).$$
Reasoning as before we have that, if at least one $z_{i+1},z_{i+2},\ldots,z_j$ is not in $T$, the probability that $z_{i+1}+z_{i+2}+\cdots+z_j=0$ is at most $\frac{2}{\ell+2}$.
Therefore we have that:
$$\mathbb{P}(z_{i+1}+z_{i+2}+\cdots+z_j=0)\leq \frac{2}{\ell+2}\mathbb{P}(\{z_{i+1},z_{i+2},\cdots,z_j\}\not\subseteq T)+$$ $$\mathbb{P}(z_{i+1},z_{i+2},\ldots,z_j\in T)\mathbb{P}(z_{i+1}+z_{i+2}+\cdots+z_j=0|z_{i+1},z_{i+2},\ldots,z_j\in T)$$
and, since $z_{i+1}+z_{i+2}+\cdots+z_j\not=0$ whenever $z_{i+1},z_{i+2},\ldots,z_j\in T$,
$$\mathbb{P}(z_{i+1}+z_{i+2}+\cdots+z_j=0)\leq \frac{2}{\ell+2}\mathbb{P}(\{z_{i+1},z_{i+2},\ldots,z_j\}\not\subseteq T).$$

Since $z_{i+1},z_{i+2},\ldots,z_j$ are at most $t$ elements randomly chosen in $S\setminus U'$, each in its prescribed coset, the probability that all of them are contained in $T$ is at least
$$\frac{\begin{pmatrix} \frac{|T|}{2} \\ \lfloor\frac{j-i}{2}\rfloor\end{pmatrix}}{\begin{pmatrix} \frac{|S\setminus U'|}{2} \\ \lfloor\frac{j-i}{2}\rfloor\end{pmatrix}}\frac{\begin{pmatrix} \frac{|T|}{2} \\ \lceil\frac{j-i}{2}\rceil\end{pmatrix}}{\begin{pmatrix} \frac{|S\setminus U'|}{2} \\ \lceil\frac{j-i}{2}\rceil\end{pmatrix}} \geq \left(\frac{\begin{pmatrix} \frac{|T|}{2} \\ \frac{t}{2}\end{pmatrix}}{\begin{pmatrix} \frac{|S\setminus U'|}{2} \\ \frac{t}{2}\end{pmatrix}} \right)^2 = \left(\frac{\begin{pmatrix} \frac{\ell}{2} \\ \frac{t}{2}\end{pmatrix}}{\begin{pmatrix} \frac{\ell+(t+2)}{2} \\ \frac{t}{2}\end{pmatrix}} \right)^2 \geq \frac{(\ell-t)^{t}}{(\ell+(t+2))^{t}}.$$
It follows that
$$\mathbb{P}(z_{i+1}+z_{i+2}+\cdots+z_j=0)\leq \frac{2}{\ell+2}\left(1-\frac{(\ell-t)^{t}}{(\ell+(t+2))^{t}}\right)$$
and hence Equation \eqref{EB2} becomes
\begin{equation}\label{EB3}\mathbb{E}(X)\leq \frac{2 t^2}{\ell+2}+\frac{2}{\ell+2}\sum_{\substack{h< i<j \leq k\\j - i \leq t}} \left(1-\frac{(\ell-t)^{t}}{(\ell+(t+2))^{t}}\right).\end{equation}
Since there are less than $(\ell+t+2)t$ pairs $(i,j)$ such that $h< i<j \leq k$ and $j - i \leq t$, from Equation \eqref{EB3} we obtain
\begin{equation}\label{EB4}\mathbb{E}(X)\leq \frac{2 t^2}{\ell+2}+\frac{2t (\ell+(t+2))}{\ell+2}\left(1-\frac{(\ell-t)^{t}}{(\ell+(t+2))^{t}}\right).\end{equation}
Here we note that the right-hand side of the last inequality goes to zero as $\ell$ goes to infinite. Therefore for $\ell$ large enough, or more precisely for $\ell\geq\bar{\ell'}$ for a suitable $\bar{\ell'}\in \mathbb{N}$, we have that $\mathbb{E}(X)<1$. This means that, if $k$ is large enough (i.e. if $k>k'_{t,\bar{\ell'}}$) there exists an ordering on which $X=0$ that is, there exists an alternating parity $t$-weak sequencing of $S$.
\endproof
\section{A Direct Construction}
In this section we will see that, for $t=3$ (and, under certain conditions, $t=4$), we can provide an explicit result in the spirit of Theorem \ref{th:main3}. More precisely, if $t=3$:
\begin{thm}
Let $G=N\rtimes_{\varphi} \Z_2$ be a group of size $n$. Then subsets of type $(k/2,k/2)$ of~$G\setminus \{(0,0)\}$ admits an alternating parity $3$-weak sequencing for any $k\geq 4$.
\end{thm}
\proof
Let $S$ be a subset of type $(k/2,k/2)$ of~$G\setminus \{(0,0)\}$.  Let assume $k \geq 5$,  then according to Proposition \ref{fix5} there is an ordering of $(k-5)$-elements of $S$ that we denote with $(y_1=(x_1, 0), y_2=(x_2, 1), \ldots, y_{k-5}=(x_{k-5}, (k-6)_2))$, such that, for any ${0\leq i<j\leq \min(k-5,i+t)}$
$$s_i=y_1 + y_2 + \cdots + y_i\not=y_1 + y_2 + \cdots + y_j=s_j\,,$$
where we use the additive notation for elements in $G$.

Now it suffices to order the last $5$ elements that we denote by:
$\{(z_{k-4}, 1), (z_{k-3}, 0), $ $(z_{k-2}, 1), (z_{k-1}, 0), (z_{k}, 1)\}$.

The element $y_{k-4}$ must be chosen in the set
$\{(z_{k-4}, 1), (z_{k-2}, 1), (z_{k}, 1)\}$ and is such that $s_{k-4}\not=s_{k-7}$. So there are at least two possible values that satisfy this relation.

The element $y_{k-3}$ must be chosen, freely in the set
$\{(z_{k-3}, 0), (z_{k-1}, 0)\}$. We will choose this element so that the remaining ones (i.e. $y_{k-2}, y_{k-1}, y_k$) do not sum to zero. This is possible since we have two choices for this value. We remark that we can make this choice even for $k=4$.

The element $y_{k-2}$ must be chosen in the set
$\{(z_{k-4}, 1), (z_{k-2}, 1), (z_{k}, 1)\}\setminus \{y_{k-4}\}$ and is such that $s_{k-2}\not=s_{k-5}$. So there is at least one possible value that satisfies this relation.

Now we have that $y_{k-1}$ is the last element of $\{(z_{k-3}, 0), (z_{k-1}, 0)\}\setminus \{y_{k-3}\}$ and $y_{k}$ is the last element of $\{(z_{k-4}, 1), (z_{k-2}, 1), (z_{k}, 1)\}\setminus \{y_{k-4},y_{k-2}\}$. Due to the choice of $y_{k-3}$, we have that $y_{k-2}+y_{k-1}+y_k\not=0$ which implies that $s_k \neq s_{k-3}$.  We remark that $s_{k-1} \neq s_{k-4}$ because their second components are different and hence the ordering is an alternating parity $3$-weak sequencing.
\endproof
If $t=4$, instead, following the line of Proposition \ref{Ramsey}, we first prove the following Ramsey-like proposition.
\begin{prop}\label{Ramsey21}
Let $S$ be a subset of type $(k/2,k/2)$ of a (not necessarily abelian) group $N\rtimes_{\varphi} \Z_2$. Then, if $k\geq 18$, $S$ contains a subset $T$ of type $(2,3)$ that does not admit zero-sum subsets of types $(2,2)$ or $(1,2)$.
\end{prop}
\proof
We proceed recursively. Given two elements $y_1=(x_1, 1)$ and $y_2=(x_2, 0)$ we have $k/2-1\geq 8$ options for $y_3=(x_3, 1)$ and, for at least $7$ of these options $y_1 + y_2 + x\not=0$. We denote by $y_3$ one of these values.

Given $y_1,y_2,y_3$ we have $k/2-1\geq 8$ options for $y_4= (x_4, 0)$ and, for at least $7$ of these options $y_1+ y_2+ y_3+ x\not=0$. We denote by $y_4$ one of these values.

Given $y_1,y_2,y_3,y_4$ we have $k/2-2\geq 7$ options for $y_5=(x_5, 1)$ and, the following system is satisfied by at least one of these options
$$
\begin{cases}
y_2+y_3+y_4+x\not=0;\\
y_2+y_1+y_4+x\not=0;\\
y_3+y_4+x\not=0;\\
y_1+y_4+x\not=0;\\
y_3+y_2+x\not=0;\\
y_1+y_2+x\not=0.\\
\end{cases}
$$
Denoted by $y_5$ this option, we have found a subset $T=\{y_1,y_2,y_3,y_4,y_5\}$ of $S$ of type $(2,3)$ that does not admit zero-sum subsets of types $(2,2)$ or $(1,2)$.
\endproof
With essentially the same proof one obtains that:
\begin{prop}\label{Ramsey3}
Let $S$ be a subset of type $((k+1)/2,(k-1)/2)$ of a (not necessarily abelian) group $N\rtimes_{\varphi} \Z_2$. Then, if $k\geq 13$, $S$ contains a subset $T$ of type $(3,2)$ that does not admit zero-sum subsets of types $(2,2)$ or $(1,2)$.
\end{prop}

Then we have the following recursive result:
\begin{prop}\label{rec1}
Let $N$ be a group of odd size and let $G=N\rtimes_{\varphi} \Z_2$. Then subsets~$S$ of type $(k/2,k/2)$ of~$G\setminus \{(0,0)\}$ admits an alternating parity $4$-weak sequencing whenever $k\geq 18$ and any subsets $S'$ of type $(k/2-5,k/2-5)$ of~$G\setminus \{(0,0)\}$ admits it.
\end{prop}
\proof
According to Proposition \ref{Ramsey21}, $S$ has a subset $T$ of type $(2,3)$ that does not admit zero-sum subsets of types $(2,2)$ or $(1,2)$.
Also, due to Proposition \ref{Ramsey3}, $S\setminus T$ has a subset $T'$ of type $(3,2)$ that does not admit zero-sum subsets of types $(2,2)$ or $(1,2)$.

Denoted by $S'=S\setminus (T\cup T')$ we note that $S'$ is of type $(k/2-5,k/2-5)$ and hence, we may assume it admits a $4$-weak sequencing $\omega'$.
We set $\omega'=(y_1=(x_1, 0), y_2=(x_2, 1),\ldots, y_{k-10}=(x_{k-10}, 1))$. Now we want to extend this ordering to the elements of~$T'$. Here we have three options in $T'$ for $y_{k-9}=(x_{k-9}, 0)$ and, for at least two of them, we have that $y_{k-12}+y_{k-11}+y_{k-10}+x\not=0$. We denote by $y_{k-9}'$ and $y_{k-9}''$ those options.

Then we have two options in $T'$ for $y_{k-8}=(x_{k-8}, 1)$ that we denote by $y_{k-8}'$ and $y_{k-8}''$ respectively. Let us assume that none of these options satisfies both
$$\begin{cases}
y_{k-11}+y_{k-10}+y_{k-9}''+x\not=0;\\
y_{k-10}+y_{k-9}''+x\not=0;
\end{cases}$$
and
$$\begin{cases}
y_{k-11}+y_{k-10}+y_{k-9}''+x\not=0;\\
y_{k-10}+y_{k-9}''+x\not=0.
\end{cases}$$
Then we must have
\begin{equation}\label{system1}\begin{cases}
y_{k-11}+y_{k-10}+y_{k-9}'+y_{k-8}'=0;\\
y_{k-10}+y_{k-9}'+y_{k-8}''=0;\\
y_{k-11}+y_{k-10}+y_{k-9}''+y_{k-8}''=0;\\
y_{k-10}+y_{k-9}''+y_{k-8}'=0.
\end{cases}\end{equation}
Equation \eqref{system1} implies that
$$\begin{cases}
y_{k-11}-y_{k-8}''+y_{k-8}'=0;\\
y_{k-11}-y_{k-8}'+y_{k-8}''=0;
\end{cases}$$
and hence that $y_{k-8}'-y_{k-8}''=-(y_{k-8}'-y_{k-8}'')$ which is not possible since $N$ has odd size and does not contain any involution.
Therefore we can find $y_{k-9}\in \{y_{k-9}', y_{k-9}''\}$ and $y_{k-8}\in \{y_{k-8}', y_{k-8}''\}$ such that
$$\begin{cases}
y_{k-11}+y_{k-10}+y_{k-9}+y_{k-8}\not=0;\\
y_{k-10}+y_{k-9}+y_{k-8}\not=0.
\end{cases}$$

Given $y_{k-9}$ and $y_{k-8}$, we have two options in $T'$ for $y_{k-7}=(x_{k-7}, 0)$ and, for at least one of them we have that $y_{k-10}+y_{k-9}+y_{k-8}+x\not=0$. We denote by $y_{k-7}$ this option.

Finally given $y_{k-9}, y_{k-8}$ and $y_{k-7}$, we have one option in $T'$ for $y_{k-6}=(x_{k-6}, 1)$ and for $y_{k-5}=(x_{k-5}, 0)$. Since $T'$ does not admit zero-sum subsets of types $(2,2)$ or $(1,2)$, we must have that:
$$\begin{cases}
y_{k-8}+y_{k-7}+y_{k-6}+y_{k-5}\not=0;\\
y_{k-9}+y_{k-8}+y_{k-7}+y_{k-6}\not=0;\\
y_{k-8}+y_{k-7}+y_{k-6}\not=0.
\end{cases}$$
It follows that we have extended the alternating parity $4$-weak ordering also to the elements of $T'$. Now we want to extend it also to the elements of $T$.

Here we have three options in $T$ for $y_{k-4}=(x_{k-4}, 1)$ and, for at least one of them we have that
$$\begin{cases} y_{k-7}+y_{k-6}+y_{k-5}+x\not=0;\\
y_{k-6}+y_{k-5}+x\not=0.\end{cases}$$
We denote by $y_{k-4}$ this option.

Given $y_{k-4}$, we have two options in $T$ for $y_{k-3}=(x_{k-3}, 0)$ and, for at least one of them we have that $y_{k-6}+y_{k-5}+y_{k-4}+x\not=0$. We denote by $y_{k-3}$ this option.

Given $y_{k-4}$ and $y_{k-3}$, we have two options in $T$ for $y_{k-2}=(x_{k-2}, 1)$ and, for at least one of them we have that $y_{k-5}+y_{k-4}+y_{k-3}+x\not=0$. We denote by $y_{k-2}$ this option. Also, since $T$ does not admit zero-sum subsets of types $(1,2)$, we must have that:
$y_{k-4}+y_{k-3}+y_{k-2}\not=0$.

Finally given $y_{k-4}, y_{k-3}$ and $y_{k-2}$, we have but one option in $T$ for $y_{k-1}=(x_{k-1}, 0)$ and for $y_{k}=(x_{k}, 1)$. Since $T$ does not admit zero-sum subsets of types $(2,2)$ or $(1,2)$, we must have that:
$$\begin{cases}
y_{k-4}+y_{k-3}+y_{k-2}+y_{k-1}\not=0;\\
y_{k-3}+y_{k-2}+y_{k-1}+y_{k}\not=0;\\
y_{k-2}+y_{k-1}+y_{k}\not=0.
\end{cases}$$
It follows that we have extended the alternating parity $4$-weak ordering of $S'$ to the elements of $T'\cup T$ and hence to the elements of $S$.
\endproof
As a consequence, we have been able to prove that:
\begin{thm}
Let $N$ be a group of odd size and let $G=N\rtimes_{\varphi} \Z_2$. Then subsets of type $(k/2,k/2)$ of~$G\setminus \{(0,0)\}$ admits an alternating parity $4$-weak sequencing whenever $k\equiv 8\pmod{10}$.
\end{thm}
\proof
We have checked the case $k=8$ with computer help, using SageMath \cite{SageMath}. Unfortunately, the cases $k=10,12,14,16$ appear to be unfeasible with our procedure.  Then the case $k\equiv 8\pmod{10}$ follows by applying, recursively, Proposition \ref{rec1}.
\endproof

\section*{Acknowledgements}
The first author was partially supported by INdAM--GNSAGA.

\end{document}